\documentclass{amsart}
\usepackage{graphicx}
\newtheorem{thm}{Theorem}[section]

\newtheorem{lem}[thm]{Lemma}
\newtheorem{prop}[thm]{Proposition}
\theoremstyle{definition}
\newtheorem{defn}[thm]{Definition}
\theoremstyle{remark} \theoremstyle{Proof}
\newtheorem{rem}[thm]{Remark}

\numberwithin{equation}{section}

\author[O. Ajebbar]{ Ajebbar Omar}
\address{ Ajebbar Omar\\Department of Mathematics\\
Ibn Zohr University, Faculty of Sciences, Agadir\\
Morocco} \email{omar-ajb@hotmail.com}
\author[E. Elqorachi]{ Elqorachi  Elhoucien}
\address{ Elqorachi Elhoucien\\Department of Mathematics\\
Ibn Zohr University, Faculty of Sciences, Agadir\\
Morocco} \email{elqorachi@hotmail.com}

\begin{document}

\title[A generalization of d'Alembert's functional equation]{A generalization of d'Alembert's functional equation on semigroups}


\begin{abstract}
Given a semigroup $S$ generated by its squares equipped with an
involutive automorphism $\sigma$ and a multiplicative function
$\mu:S\to\mathbb{C}$ such that $\mu(x\sigma(x))=1$ for all $x\in S$,
we determine the complex-valued solutions of the following
functional equation
\begin{equation*}f(xy)-\mu(y)f(\sigma(y)x)=g(x)h(y),\quad x,y\in S,\end{equation*}
\end{abstract}
\maketitle
\section{Introduction}
The identity
$$\cos(x+y)+\cos(x-y)=2\cos x\cos y,\,x,y\in\mathbb{R},$$ also called
the cosine functional equation, is a starting point of d'Alembert's
functional equation
\begin{equation}\label{EQ01}g(x+y)+g(x-y)=2g(x)g(y),\,x,y\in\mathbb{R},\end{equation}
where $g:\mathbb{R}\to\mathbb{C}$ is the unknown function. The
functional equation (\ref{EQ01}) go back to d'Alembert's
investigations \cite{d'Albert0}, \cite{d'Albert1} and
\cite{d'Albert2}. It is well known that the continuous solutions
$g\neq0$ of (\ref{EQ01}) are the functions $g(x)=\dfrac{e^{\alpha
x}+e^{-\alpha x}}{2},\,x\in\mathbb{R}$, where $\alpha\in\mathbb{C}$
is a constant.\\The obvious extension of (\ref{EQ01}) from
$\mathbb{R}$ to an abelian group $(G,+)$ is the functional equation
\begin{equation}\label{EQ02}g(x+y)+g(x-y)=2g(x)g(y),\,x,y\in G,\end{equation}
where $g:\mathbb{G}\to\mathbb{C}$ is the unknown function. The
nonzero solutions of (\ref{EQ02}) are the functions of the form
$g(x)=\dfrac{\chi(x)+\chi(-x)}{2},\,x\in G$, where $\chi$ is a
character of $G$. The functional equation
\begin{equation}\label{EQ03}g(xy)+g(x\tau(y))=2g(x)g(y),\,x,y\in {G},\end{equation}
where $G$ is a group that need not be abelian, $\tau:G\to G$ is an
involution (i.e. $\tau(xy)=\tau(y)\tau(x)$ and $\tau(\tau(x))=x$ for
all $x,y\in G$)  is a generalization of (\ref{EQ02}).\\
The first result for non-abelian groups was obtained by Kannappan
\cite{Kannappan}. With $\tau(x)=x^{-1}$ for all $x\in G$, the
abelian solutions of equation  (\ref{EQ03}) (i.e. those satisfy the
Kannappan condition: $g(xyz)=g(yxz)$ for all $x,y,z\in G$) are of
the form $g(x)=\dfrac{m(x)+m(-x)}{2},\,x\in G$ where
$m:G\to\mathbb{C}$ is a multiplicative function. The solutions of
(\ref{EQ03}) were obtained by Davison \cite{Davison} for general
groups, even monoids.
\\ The complex valued-solutions of the
following variant of d'Alembert's functional equation
\begin{equation*}f(xy)+f(\sigma(y)x)=2f(x)f(y),\;x,y\in S,\end{equation*}
where $S$ is a semigroup and $\sigma:S\to S$ is an involutive
automorphism, was determined by Stetk{\ae}r \cite{Stetkaer3}.\\
At the same time Ebanks and Stetk{\ae}r \cite{Ebanks and Stetkaer2}
determined the complex-valued solutions of the functional equation
\begin{equation*}f(xy)-f(\sigma(y)x)=g(x)h(y)\;x,y\in M,\end{equation*}
where $M$ is a monoid generated by its squares and $\sigma:M\to M$
is an involutive automorphism.\\Bouikhalene and Elqorachi
\cite{Bouikhalene and Elqorachi} solved the functional equation
\begin{equation}\label{EQ1}f(xy)-\mu(y)f(\sigma(y)x)=g(x)h(y),\,x,y\in M,\end{equation}
where $M$ is a monoid generated by its squares or a group,
$\sigma:M\to M$ is an involutive automorphism and
$\mu:M\to\mathbb{C}$ is a multiplicative function such that
$\mu(x\sigma(x))=1$ for all $x\in M$.\\For more details about
(\ref{EQ03}) we refer to \cite{Parnami et al.}, \cite[Proposition
2.11]{Davison}, \cite[Lemme IV.4]{Stetkaer4} and \cite[Proposition
4.2]{Yang}.\\The identity element of the monoid or the group is used
in the proofs of \cite[Theorem 2.1, Proposition 2.2, Theorem
3.2]{Bouikhalene and Elqorachi} and \cite[ Proposition 4.1, Theorem
4.3]{Ebanks and Stetkaer2}.\\
Similar functional equations were treated in Chapter 13 of the book
\cite{Acz�l} by Acz\'{e}l and Dhombres.\\The present paper shows
that the identity element is not crucial by given proofs in the
setting of general semigroups. Emphasize that $\sigma$ in the
present paper is a homomorphism, not an anti-homomorphism like the
group inversion found in other papers.
\\Although we use similar computations to the ones in \cite{Bouikhalene and Elqorachi} and \cite{Ebanks and
Stetkaer2} the general setting of (\ref{EQ1}) is for $S$ to be a
semigroup generated by its squares, because to formulate the
functional equation equation (\ref{EQ1}) and obtain some key
properties of solutions like centrality and parity we need only an
associative composition in $S$ and the assumption that the semigroup
$S$ is generated by its squares, not an identity element and
inverses. The main new feature of the present paper is that we do
not assume the underlying semigroup has an identity. That makes the
exposition more involved and explains why our proofs are longer that
those of previous papers about the same functional equations. Thus
we study in the present paper (\ref{EQ1}) extending works in which
$S$ is a group or a monoid.
\par The organization of the paper is as follows. In the next
section we give notations and terminology. In the third section we
give preliminary results that we need in the paper. In section 4 we
prove our main results. \par The explicit formulas for the solutions
are expressed in terms of multiplicative and additive functions.
\section{Notations and Terminology}
\begin{defn} Let $f:S\to\mathbb{C}$ be a function on a
semigroup $S$. We say that\\
$f$ is additive if $f(xy)=f(x)+f(y)$ for all $x,y\in S$.\\
$f$ is multiplicative if $f(xy)=f(x)f(y)$ for all $x,y\in S$.\\
$f$ is central if $f(xy)=f(yx)$ for all $x,y\in S$.
\end{defn}
If $S$ is a semigroup, $\sigma:S\to S$ an involutive automorphism
and $\mu:S\to\mathbb{C}$ a multiplicative function such that
$\mu(x\sigma(x))=1$ for all $x\in S$, we define the nullspace
$$\mathcal{N}_{\mu}(\sigma,S):=\{\theta:S\to\mathbb{C}\mid
\theta(xy)-\mu(y)\theta(\sigma(y)x)=0,\,x,y\in S\}.$$If
$\chi:S\to\mathbb{C}$ is a multiplicative function and $\chi\neq0$,
then $I_{\chi}:=\{x\in S\mid\chi(x)=0\}$ is either empty or a proper
subset of $S$. $I_{\chi}$ is a two sided ideal in $S$ if not empty
and $S\setminus I_{\chi}$ is a subsemigroup of $S$. Notice that
$I_{\chi}$ is also a subsemigroup of $S$.\\
For any function $F:S\to\mathbb{C}$ we define the function
$$F^{*}(x)=\mu(x)F(\sigma(x)),\,x\in S.$$
Let $f:S\to\mathbb{C}$. We call $f^{e}:=\dfrac{f+f^{*}}{2}$ the even
part of $f$ and
$f^{o}:=\dfrac{f-f^{*}}{2}$ its odd part. The function $f$ is said to be even if $f=f^{*}$, and $f$ is said to be odd if $f=-f^{*}$.\\
If $g,h:S\to\mathbb{C}$ are two functions we define the function
$(g\otimes h)(x,y):=g(x)h(y),\,x,y\in S$.\\ \textbf{Blanket
assumption:} Throughout this paper $S$ denotes a semigroup (a set
with an associative composition) generated by its squares. The map
$\sigma:S\to S$ denotes an involutive automorphism. That $\sigma$ is
involutive means that $\sigma(\sigma(x))=x$ for all $x\in S$. We
denote by $\mu:S\to\mathbb{C}$ a multiplicative function such that
$\mu(x\sigma(x))=1$ for all $x\in S$.
\section{$\mu$-sine subtraction law on a semigroup generated by its squares}
In this section we extend the results obtained in \cite[Theorem 2.1,
Proposition 2.2]{Bouikhalene and Elqorachi}, \cite[Theorem
4.12]{Stetkaer2} and \cite[Theorem 3.2 and Lemma 3.4]{Ebanks and
Stetkaer2} on groups and monoids to semigroups generated by their
squares by solving the $\mu$-sine subtraction law
\begin{equation}\label{Eq6-1}\mu(y)k(x\sigma(y))=k(x)l(y)-k(y)l(x),\,x,y\in
S.\end{equation}
\begin{lem}\label{lem61}Let $G$ be a semigroup such that $G=\{xy\,|\,x,y\in
G\}$. Let $f$ and $F$ be functions on $G$ such that $f(xy)=F(yx)$
for all $x,y\in G$. Then $f=F$. In particular $f$ is central.
\end{lem}
\begin{proof} For any $x,y,z\in G$ we have
$f(xyz)=f((xy)z)))=F(zxy)=F((zx)y)=f(y(zx))=f((yz)x)=F(x(yz))=F(xyz)$.
Applying the assumption on $G$ twice we see that any element of $G$
can be written as $xyz$, where $x,y,z\in G$. So $f=F$ and then $f$
is central. This finishes the proof.
\end{proof}
\begin{rem}\label{rem1} The proof of Lemma \ref{lem61} works also
for if $G$ is regular semigroup, which by definition means that for
each $a\in G$ there exists an element $x\in G$ such that $a=axa$.
\end{rem}
\begin{prop}\label{prop61} The solutions of the functional equation
(\ref{Eq6-1}) with $k\neq0$ are the followings pairs\\
(1) $k=c_{1}\frac{\chi-\chi^{*}}{2}$,
$l=\frac{\chi+\chi^{*}}{2}+c_{2}\frac{\chi-\chi^{*}}{2}$ where
$\chi:S\to\mathbb{C}$ is a multiplicative function and
$c_{1}\in\mathbb{C}\setminus\{0\}$, $c_{2}\in\mathbb{C}$ are
constants such that $\chi^{*}\neq\chi$,\\
(2) \[ \left\{
\begin{array}{r c l}
k&=&\chi\,A,\quad$$l=\chi(1+cA)$$\quad\text{on}\quad$$S\setminus I_{\chi},$$\\
k&=&0,\quad$$l=0$$\quad\text{on}\quad$$I_{\chi},$$\\
\end{array}
\right.
\]
where $c\in\mathbb{C}$ is a constant, $\chi:S\to\mathbb{C}$ is a
nonzero multiplicative function and $A:S\setminus
I_{\chi}\to\mathbb{C}$ is a nonzero additive function such that
$\chi^{*}=\chi$ and $A\circ\sigma=-A$.
\end{prop}
\begin{proof} Let $x,y\in S$ be arbitrary. By interchanging $x$ and $y$ in (\ref{Eq6-1})
we get the identity $\mu(y)k(x\sigma(y))=-\mu(x)k(y\sigma(x))$,
which, applied to the pair $(x,\sigma(y))$, read
$\mu(\sigma(y))k(xy)=-\mu(x)k(yx)$. Multiplying this by $\mu(y)$ and
using that $\mu:S\to\mathbb{C}$ is a multiplicative function and
that $\mu(y\sigma(y))=1$, we get that $k(xy)=-k^{*}(yx)$. So, $x$
and $y$ being arbitrary, we deduce, according to Lemma \ref{lem61},
that
$k=-k^{*}$ and $k$ is central.\\
On the other hand, by using the same computation used by Bouikhalene
and Elqorachi in the proof of \cite[Theorem 2.1]{Bouikhalene and
Elqorachi} we get that there exists a constant $c\in\mathbb{C}$ such
that
\begin{equation}\label{eq7-2}l^{o}=ck\end{equation}
and that the pair $(k,l^{e})$ satisfies the sine addition law
\begin{equation*}k(xy)=k(x)l^{e}(y)+k(y)l^{e}(x),\,x,y\in S,\end{equation*}
hence, according to \cite[Proposition 1.1]{Bouikhalene and
Elqorachi}, that
pair falls into two categories:\\
(i) $k=c_{1}\frac{\chi_{1}-\chi_{2}}{2}$ and
$l^{e}=\frac{\chi_{1}+\chi_{2}}{2}$, where
$\chi_{1},\chi_{2}:S\to\mathbb{C}$ are different multiplicative
functions and $c_{1}\in\mathbb{C}\setminus\{0\}$ is a constant.
Since $k^{*}=-k$, $\mu l^{e}=l^{e}$, and $c_{1}\neq0$ a small
computation shows that $\chi_{2}=\chi_{1}^{*}$. Defining
$\chi:=\chi_{1}$ and $c_{2}:=c_{1}c\in\mathbb{C}$, and using that
$l=l^{e}+l^{o}$ we get that
$l=\frac{\chi+\chi^{*}}{2}+c_{2}\frac{\chi-\chi^{*}}{2}$ , where
$\chi:S\to\mathbb{C}$ is a multiplicative function such that
$\chi\neq\chi^{*}$ and $c_{1}\in\mathbb{C}\setminus\{0\}$,
$c_{2}\in\mathbb{C}$ are constants. The result occurs
in (1) of Proposition \ref{prop61}.\\
(ii) \[ \left\{
\begin{array}{r c l}
k&=&\chi\,A,\quad$$l^{e}=\chi$$\quad\text{on}\quad$$S\setminus I_{\chi},$$\\
k&=&0,\quad$$l^{e}=0$$\quad\text{on}\quad$$I_{\chi},$$\\
\end{array}
\right.
\]
where $\chi:S\to\mathbb{C}$ is a nonzero multiplicative function and
$A:S\setminus I_{\chi}\to\mathbb{C}$ is a nonzero additive function.
So, taking (\ref{eq7-2}) into account, we get that $l^{o}=c\chi\,A$
on $S\setminus I_{\chi}$ and $l^{o}=0$ on $I_{\chi}$. Hence
\[ \left\{
\begin{array}{r c l}
l&=&\chi+c\chi\,A=\chi(1+cA)\quad\text{on}\quad$$S\setminus I_{\chi},$$\\
l&=&0,\quad\text{on}\quad$$I_{\chi}.$$\\
\end{array}
\right.
\]
Notice that $l^{e}(x)=\chi(x)$ for all $x\in S$. As $\mu
l^{e}\circ\sigma=l^{e}$ we have $\chi^{*}=\chi$ and then
$\sigma(S\setminus I_{\chi})=S\setminus I_{\chi}$. So, using that
$k^{*}=-k$ and the fact that $\chi(x)\neq0$ for all $x\in S\setminus
I_{\chi}$, we get that $A\circ\sigma=-A$. The result occurs in part
(2) of Proposition \ref{prop61}.
\par Conversely, if $k$ and $l$ are of the forms (1)-(2) in Proposition
(\ref{prop61}) we check by elementary computations that the pair
$(k,l)$ is a solution of Eq. (\ref{Eq6-1}) with $k\neq0$. This
completes the proof of Proposition \ref{prop61}.
\end{proof}
\begin{rem}\label{rem2} \cite[Proposition 1.1]{Bouikhalene and
Elqorachi} corresponds to \cite[Lemma 3.4]{Ebanks and Stetkaer2}.
So, like in \cite[Remark 3.5]{Ebanks and Stetkaer2}, the Proposition
\ref{prop61} is also valid if $S=\{xy\,|\,x,y\in S\}$, or if $S$ is
a regular semigroup.
\end{rem}
\section{Solutions of Eq. (\ref{EQ1}) on a semigroup generated by its squares}
The following Lemma will be used later.
\begin{lem}\label{lem621} Let $G$ be a semigroup, $\tau:G\to G$ be an involutive automorphism and $\nu:G\to\mathbb{C}$ be a
multiplicative function such that $\nu(x\tau(x))=1$ for all $x\in
G$, and
$\theta\in\mathcal{N}_{\nu}(\tau,G)$.\\
(1) $\theta=\theta^{*}$ on $\{xy\,|\,x,y\in G\}$, i.e., $\theta$ is
even on products.\\
(2) If $\theta$ is odd, i.e., $\theta=-\theta^{*}$, then $\theta=0$
on $\{xy\,|\,x,y\in G\}$. In particular $\theta=0$ if $G$ is a
regular semigroup, or if $G$ is generated by its squares.
\end{lem}
\begin{proof}
(1) For all $x,y\in G$ we have
$$\theta(xy)=\nu(y)\theta(\tau(y)x)=\nu(y)\nu(x)\theta(\tau(x)\tau(y))=\nu(xy)\theta(\tau(xy))=\theta^{*}(xy),$$
which proves (1).\\
(2) If $\theta$ is odd then $\theta=-\theta^{*}$. So, according to
Lemma \ref{lem621}(1), $\theta=0$ on $\{xy\,|\,x,y\in G\}$. If $G$
is a regular semigroup or if $G$ is generated by its squares, then
any element of $G$ can be written as $xy$ where $x,y\in G$. Hence
$\theta=0$. This finishes the proof.
\end{proof}
In Lemma \ref{lem62} below we give some key properties of solutions
of the functional equation (\ref{EQ1}).
\begin{lem}\label{lem62} Let $f,g,h:S\to\mathbb{C}$ be a
solution of the functional equation (\ref{EQ1}). Suppose that
$g\neq0$ and $h\neq0$. Then\\
(1) There exists a function $l:S\to\mathbb{C}$ such that
\begin{equation}\label{eq72-1}h(xy)=g^{*}(x)l(y)-g(y)l^{*}(x)\end{equation} for all $x,y\in
S$.\\
(2) $h^{*}=-h$ and $h$ is central.\\
(3)
\begin{equation}\label{eq72-2}f^{e}(xy)-\mu(y)f^{e}(\sigma(y)x)=g^{o}(x)h(y)\end{equation}
for all $x,y\in S$.\\
(4)
\begin{equation}\label{eq72-3}f^{o}(xy)-\mu(y)f^{o}(\sigma(y)x)=g^{e}(x)h(y)\end{equation}
for all $x,y\in
S$.\\
(5)
\begin{equation}\label{eq72-4}2f^{o}(xy)=g^{e}(x)h(y)+g^{e}(y)h(x)\end{equation}
for all $x,y\in
S$. In particular $f^{o}$ is central.\\
(6) There exists a constant $b\in\mathbb{C}$ such that $g^{o}=bh$.
\end{lem}
\begin{proof} (1) We use similar computations to those of the proof of Proposition 3 in \cite{Ebanks and Stetkaer2}.\\
Let $x,y,z\in S$ be arbitrary. By applying Eq (\ref{EQ1}) to the
pairs $(x,yz)$, $(\sigma(y),\sigma(z)x)$ and $(z,\sigma(xy))$ we
obtain
\begin{equation}\label{eq7-4}f(xyz)-\mu(yz)f(\sigma(yz)x)=g(x)h(yz),\end{equation}
\begin{equation}\label{eq7-5}f(\sigma(yz)x)-\mu(\sigma(z)x)f(z\sigma(xy))=g(\sigma(y))h(\sigma(z)x),\end{equation}
\begin{equation}\label{eq7-6}f(z\sigma(xy))-\mu(\sigma(xy))f(xyz)=g(z)h(\sigma(xy)).\end{equation}
By multiplying (\ref{eq7-5}) by $\mu(yz)$, and (\ref{eq7-6}) by
$\mu(xy)$ we get that
\begin{equation}\label{eq7-7}\mu(yz)f(\sigma(yz)x)-\mu(xy)f(z\sigma(xy))=\mu(yz)g(\sigma(y))h(\sigma(z)x),\end{equation}
\begin{equation}\label{eq7-8}\mu(xy)f(z\sigma(xy))-f(xyz)=\mu(xy)g(z)h(\sigma(xy)).\end{equation}
By adding (\ref{eq7-4}), (\ref{eq7-7}) and (\ref{eq7-8}) we obtain
\begin{equation}\label{eq7-9}g(x)h(yz)+\mu(yz)g(\sigma(y))h(\sigma(z)x)+\mu(xy)g(z)h(\sigma(xy))=0.\end{equation}
As $g\neq0$ we chose $x_{0}\in S$ such that $g(x_{0})\neq0$. So, by
putting $x=x_{0}$ we deduce from (\ref{eq7-9}) that
\begin{equation*}\begin{split}h(yz)&=-g(x_{0})^{-1}\mu(yz)g(\sigma(y))h(\sigma(z)x_{0})-g(x_{0})^{-1}\mu(x_{0}y)g(z)h(\sigma(x_{0}y))\\
&=g^{*}(y)(-g(x_{0})^{-1}\mu(z)h(\sigma(z)x_{0}))+g(z)(-g(x_{0})^{-1}\mu(x_{0}y)h(\sigma(x_{0}y))).\end{split}\end{equation*}
So, $y$ and $z$ being arbitrary, we get from the last identity, by
defining the functions $l(z):=-g(x_{0})^{-1}\mu(z)h(\sigma(z)x_{0})$
and $l_{1}(z):=-g(x_{0})^{-1}\mu(x_{0}z)h(\sigma(x_{0}z))$ for all
$z\in S$, that
\begin{equation}\label{eq7-10}h(yz)=g^{*}(y)l(z)+g(z)l_{1}(y)\end{equation}
for all $y,z\in S$.\\
Using (\ref{eq7-10}) and (\ref{eq7-9}) we obtain
\begin{equation}\label{eq7-11}\begin{split}0&=g(x)[g^{*}(y)l(z)+g(z)l_{1}(y)]+\mu(yz)g(\sigma(y))[\mu(\sigma(z))g(z)l(x)\\&+g(x)l_{1}(\sigma(z))]
+\mu(xy)g(z)[\mu(\sigma(x))g(x)l(\sigma(y))+g(\sigma(y))l_{1}(\sigma(x))]\\
&=g(x)g(\sigma(y))[\mu(y)l(z)+\mu(yz)l_{1}(\sigma(z))]+g(x)g(z)[l_{1}(y)+l^{*}(y)]\\
&+g(\sigma(y))g(z)[\mu(y)l(x)+\mu(xy)l_{1}(\sigma(x))],\end{split}\end{equation}
for all $x,y,z\in S$. By putting $x=x_{0}$ and $y=\sigma(x_{0})$ in
the identity above we get that
\begin{equation*}g(x_{0})^{2}\mu(\sigma(x_{0}))[l(z)+\mu(z)l_{1}(\sigma(z))]+2g(x_{0})g(z)[l_{1}(\sigma(x_{0}))+\mu(\sigma(x_{0}))l(x_{0}))]=0.\end{equation*}
As $g(x_{0})\mu(x_{0})\neq0$ we deduce from the identity above that
there exists a constant $c\in\mathbb{C}$ such that
\begin{equation}\label{eq7-12}l(z)+\mu(z)l_{1}(\sigma(z))=cg(z),\end{equation}
for all $z\in S$. From (\ref{eq7-11}) and (\ref{eq7-12}) we get, by
a small computation, that
$$3cg(x)g(\sigma(y))g(z)=0$$ for all $x,y,z\in S$, which implies that
$c=0$ because $g\neq0$. Then, using that $\mu(z\sigma(z))=1$ for all
$z\in S$, we infer from (\ref{eq7-12}), that $l_{1}=-l^{*}$. Hence,
(\ref{eq7-10}) becomes
$$h(yz)=g^{*}(y)l(z)-g(z)l^{*}(y)$$
for all $y,z\in S$, which occurs in (1) Lemma \ref{lem62}.\\
(2) Let $x,y\in S$ be arbitrary. By applying (\ref{eq72-1}) to the
pair $(\sigma(y),\sigma(x))$ and multiplying the identity obtained
by $\mu(xy)$ we get that
$$\mu(xy)h\circ\sigma(yx)=g(y)l^{*}(x)-g^{*}(x)l(y)=-h(xy).$$
Proceeding exactly as in the proof of the Proposition \ref{prop61}
we deduce that $h^{*}=-h$ and $h$ is central. This is the
result (2) of Lemma \ref{lem62}.\\
(3) Let $x,y\in S$ be arbitrary. We have
\begin{equation*}\begin{split}&f^{e}(xy)-\mu(y)f^{e}(\sigma(y)x)=\frac{f(xy)+\mu(xy)f(\sigma(x)\sigma(y))}{2}\\&-\mu(y)\frac{f(\sigma(y)x)+\mu(\sigma(y)x)f(y\sigma(x)}{2}\\
&=\frac{f(xy)-\mu(y)f(\sigma(y)x)}{2}+\mu(xy)\frac{f(\sigma(x)\sigma(y))-\mu(\sigma(y))f(y\sigma(x))}{2}\\
&=\frac{g(x)h(y)}{2}+\frac{g^{*}(x)\mu(y)h(\sigma(y))}{2},\end{split}\end{equation*}
which implies, since $h^{*}=-h$, that
$f^{e}(xy)-\mu(y)f^{e}(\sigma(y)x)=\dfrac{[g(x)-g^{*}(x)]h(y)}{2}$.
So that $f^{e}(xy)-\mu(y)f^{e}(\sigma(y)x)=g^{o}(x)h(y)$ for all
$x,y\in S$. This is part (3).\\
(4) By subtracting (\ref{eq72-2}) from (\ref{EQ1}) a small
computation shows that
$$f^{o}(xy)-\mu(y)f^{o}(\sigma(y)x)=g^{e}(x)h(y)$$ for all $x,y\in
S$. This is the result (4).\\
(5) By applying (\ref{eq72-3}) to the pair $(\sigma(y),x)$ and
multiplying the identity obtained by $\mu(y)$ we get that
$$\mu(y)f^{o}(\sigma(y)x)-\mu(xy)f^{o}(\sigma(x)\sigma(y))=\mu(y)g^{e}(\sigma(y))h(x).$$
Since $\mu f^{o}\circ\sigma=-f^{o}$ and $\mu g^{e}\circ\sigma=g^{e}$
the identity above implies that
\begin{equation*}\mu(y)f^{o}(\sigma(y)x)+f^{o}(xy)=g^{e}(y)h(x).\end{equation*}
When we add this to (\ref{eq72-3}) we obtain the functional equation
(\ref{eq72-4}).\\
(6) By applying (\ref{eq72-2}) to the pair $(x,\sigma(y))$ and
multiplying the identity obtained by $\mu(y)$, and using the fact
that $\mu(y\sigma(y))=1$ and $h^{*}=-h$, we get that
\begin{equation*}\mu(y)f^{e}(x\sigma(y))-f^{e}(yx)=-g^{o}(x)h(y).\end{equation*}
The identity (\ref{eq72-2}) applied to the pair $(y,x)$ gives
\begin{equation*}f^{e}(yx)-\mu(x)f^{e}\circ\sigma(x\sigma(y))=g^{o}(y)h(x).\end{equation*}
By adding the two last identities, and seeing that $\mu
f^{e}\circ\sigma=f^{e}$ and $\mu(y\sigma(y))=1$ for all $y\in S$, we
obtain
\begin{equation*}\begin{split}&g^{o}(y)h(x)-g^{o}(x)h(y)=\mu(y)f^{e}(x\sigma(y))-\mu(x)f^{e}\circ\sigma(x\sigma(y))\\
&=\mu(y)f^{e}(x\sigma(y))-\mu(x)\frac{f^{e}(x\sigma(y))}{\mu(x\sigma(y))}\\
&=\mu(y)f^{e}(x\sigma(y))-\mu(y)f^{e}(x\sigma(y))\\
&=0\end{split}\end{equation*} for all $x,y\in S$. So $g^{o}$ and $h$
are linearly dependent. As $h\neq0$ by assumption, we get that there
exists a constant $b\in\mathbb{C}$ such that $g^{o}=bh$, which
occurs in part (6). This completes the proof of Lemma \ref{lem62}.
\end{proof}
\par In theorem \ref{thm63} we solve the functional equation (\ref{EQ1}) on
semigroups generated by their squares.
\begin{thm}\label{thm63}The solutions $f,g,h:S\to\mathbb{C}$ of the functional equation
(\ref{EQ1}) can be listed as
follows:\\
(1) $f=\theta$, $g=0$ and $h$ is arbitrary, where $\theta\in\mathcal{N}_{\mu}(\sigma,S)$.\\
(2)$f=\theta$, $g$ is arbitrary and $h=0$, where $\theta\in\mathcal{N}_{\mu}(\sigma,S)$. \\
(3)
$f=\theta+\alpha\frac{\chi+\chi^{*}}{2}+\beta\frac{\chi-\chi^{*}}{2}$
and $g\otimes
h=2[\beta\frac{\chi+\chi^{*}}{2}+\alpha\frac{\chi-\chi^{*}}{2}]\otimes\frac{\chi-\chi^{*}}{2}$
where $\alpha,\beta\in\mathbb{C}$ are constants,
$\theta\in\mathcal{N}_{\mu}(\sigma,S)$ and $\chi:S\to\mathbb{C}$ is
a multiplicative function such that $(\alpha,\beta)\neq(0,0)$ and
$\chi^{*}\neq\chi$.\\
(4)\[ \left\{
\begin{array}{r c l}
f&=&\theta+\frac{\alpha}{2}\chi\,A+\frac{\beta}{4}\chi\,A^{2},\quad$$g=\alpha\chi+\beta\chi\,A$$,
\quad$$h=\chi\,A$$\quad\text{on}\quad$$S\setminus I_{\chi},$$\\
f&=&\theta,\quad$$g=h=0$$\quad\text{on}\quad$$I_{\chi},$$\\
\end{array}
\right.
\]
where $\alpha,\beta\in \mathbb{C}$ are constants,
$\theta\in\mathcal{N}_{\mu}(\sigma,S)$, $\chi:S\to\mathbb{C}$ is a
nonzero multiplicative function and $A:S\setminus
I_{\chi}\to\mathbb{C}$ is a nonzero additive function such that
$(\alpha,\beta)\neq(0,0)$,
$\chi^{*}=\chi$ and $A\circ\sigma=-A$.\\
\end{thm}
\begin{proof} We check by elementary computations that if $f,g$ and $h$
are of the forms (1)-(4) then $(f,g,h)$ is a solution of (\ref{EQ1})
, so left is that any solution $(f,g,h)$ of (\ref{EQ1}) fits into
(1)-(4).\\Let $f,g,h:S\to\mathbb{C}$ satisfy the functional equation
(\ref{EQ1}). The parts (1)and (2) are trivial. So, in what remains
of the proof we assume that $g\neq0$ and $h\neq0$. According to
Lemma \ref{lem62}(6) there exists a constant $b\in\mathbb{C}$ such
that
\begin{equation}\label{eq7-13}g^{o}=bh.\end{equation} We split the discussion into
the cases $b\neq0$ and $b=0$.\\
\underline{Case A}: $b\neq0$. Then $g^{0}\neq0$. From (\ref{eq72-2})
we get that the triple $(f^{e},g^{o},h)$ satisfies the functional
equation (\ref{EQ1}). Hence, according to Lemma \ref{lem62}(1),
there exists a function $L:S\to\mathbb{C}$ such that
\begin{equation*}h(xy)=\mu(x)g^{o}(\sigma(x))L(y)-g^{o}(y)L^{*}(x),\end{equation*}
for all $x,y\in S$. Since $\mu g^{o}\circ\sigma=-g^{o}$ we get,
taking (\ref{eq7-13}) into account, that
\begin{equation*}h(xy)=-bh(x)L(y)-b\mu(x)h(y)L(\sigma(x)).\end{equation*}
Applying this to the pair $(x,\sigma(y))$ and multiplying the
identity obtain by $\mu(y)$, and taking into account that $\mu
h\circ\sigma=-h$, we obtain
\begin{equation*}\mu(y)h(x\sigma(y))=h(x)(-bL^{*}(y))-h(y)(-bL^{*}(x)),\end{equation*}
for all $x,y\in S$. Defining $m:=-bL^{*}$ we deduce from the
identity above that the pair $(h,m)$ satisfies the $\mu$-subtraction
law
\begin{equation}\label{eq7-14}\mu(y)h(x\sigma(y))=h(x)m(y)-h(y)m(x),\end{equation}
for all $x,y\in S$.\\
On the other hand, the identity (\ref{eq72-3}) implies that the
triple $(f^{o},g^{e},h)$ satisfies the functional equation
(\ref{EQ1}). Let $x,y,z\in S$ be arbitrary. As in the proof of Lemma
\ref{lem62} we deduce that the pair $(g^{e},h)$ satisfies the
identity (\ref{eq7-9}), i.e.,
\begin{equation*}g^{e}(x)h(yz)+\mu(yz)g^{e}(\sigma(y))h(\sigma(z)x)+\mu(xy)g^{e}(z)h(\sigma(xy))=0,\end{equation*}
for all $x,y,z\in S$. Sine $\mu g^{e}\circ\sigma=g^{e}$ and
$h^{*}=-h$ and $h$ is central, we get from the identity above that
\begin{equation*}g^{e}(x)h(zy)=g^{e}(z)h(xy)-\mu(z)g^{e}(y)h(x\sigma(z)),\end{equation*}
which implies, by replacing $y$ by $\sigma(y)$, multiplying both
sides by $\mu(y)$ and using that $\mu g^{e}\circ\sigma=g^{e}$, that
\begin{equation*}g^{e}(x)\mu(y)h(z\sigma(y))=g^{e}(z)\mu(y)h(x\sigma(y))-g^{e}(y)\mu(z)h(x\sigma(z)).\end{equation*}
Using that the pair $(h,m)$ satisfies the $\mu$-subtraction law we
get from the last identity that
$g^{e}(x)[h(z)m(y)-h(y)m(z)]=g^{e}(z)[h(x)m(y)-h(y)m(x)]-g^{e}(y)[h(x)m(z)-h(z)m(x)]$.
So that
\begin{equation}\label{eq7-15}\begin{split}&g^{e}(x)[h(z)m(y)-h(y)m(z)]=[g^{e}(z)m(y)-g^{e}(y)m(z)]h(x)\\
&+[g^{e}(y)h(z)-g^{e}(z)h(y)]m(x).\end{split}\end{equation} Now, let
$y,z\in S$ be arbitrary in (\ref{eq7-15}). Since the function
$g^{e}$ is even and $h$ is odd, we get from (\ref{eq7-15}) that
\begin{equation}\label{eq7-16}g^{e}(x)[h(z)m(y)-h(y)m(z)]=[g^{e}(y)h(z)-g^{e}(z)h(y)]m^{e}(x)\end{equation}
for all $x\in S$. Since $h\neq0$ we get, seeing the formulas of the
solutions of the $\mu$-subtraction law given by the Proposition
\ref{prop61}, that the functions $h$ and $m$ are linearly
independent.
So that there exist $y_{0},z_{0}\in S$ such that $h(z_{0})m(y_{0})-h(y_{0})m(z_{0})\neq0$.\\
By putting $(y,z)=(y_{0},z_{0})$ in (\ref{eq7-16}) we deduce that
there exists a constant $c_{3}\in\mathbb{C}$ such that
\begin{equation}\label{eq7-17}g^{e}=c_{3}m^{e}.\end{equation}
Moreover, taking into account (\ref{eq7-14}), we get from
Proposition \ref{prop61} that the pair $(h,m)$ falls into two
categories (i) and (ii) that we deal with separately:\\
(i)
\begin{equation}\label{eq7-18}h=c_{1}\frac{\chi-\chi^{*}}{2}\end{equation}
and
\begin{equation}\label{eq7-19}m=\frac{\chi+\chi^{*}}{2}+\alpha\frac{\chi-\chi^{*}}{2},\end{equation}
where $c_{1}\in\mathbb{C}\setminus\{0\}$, $\alpha\in\mathbb{C}$ are
constants and $\chi:S\to\mathbb{C}$ is a multiplicative function
such that $\chi^{*}\neq\chi$. Then $m^{e}=\frac{\chi+\chi^{*}}{2}$.
So, in view of (\ref{eq7-13}) and (\ref{eq7-17}), we get that
\begin{equation}\label{eq7-20}g^{o}=c_{2}\frac{\chi-\chi^{*}}{2},\end{equation}
and
\begin{equation}\label{eq7-21}g^{e}=c_{3}\frac{\chi+\chi^{*}}{2},\end{equation} where $c_{2}:=bc_{1}\in\mathbb{C}\setminus\{0\}$ is a
constant. Hence
\begin{equation}\label{eq7-22}g=c_{3}\frac{\chi+\chi^{*}}{2}+c_{2}\frac{\chi-\chi^{*}}{2}.\end{equation}
By substituting (\ref{eq7-18}) and (\ref{eq7-20}) in (\ref{eq72-2})
we get that
\begin{equation*}\begin{split}&f^{e}(xy)-\mu(y)f^{e}(\sigma(y)x)=\frac{c_{1}c_{2}}{4}[\chi(x)-\mu(x)\chi(\sigma(x))][\chi(y)-\mu(y)\chi(\sigma(y))]\\
&=\frac{c_{1}c_{2}}{4}[\chi(xy)-\mu(y)\chi(\sigma(y)x)-\mu(x)\chi(\sigma(x)y)+\mu(xy)\chi(\sigma(xy))]\\
&=\frac{c_{1}c_{2}}{4}(\chi+\chi^{*})(xy)-\mu(y)\frac{c_{1}c_{2}}{4}(\chi+\chi^{*})(\sigma(y)x),\end{split}\end{equation*}
for all $x,y\in S$. So that
\begin{equation*}(f^{e}-\frac{c_{1}c_{2}}{4}(\chi+\chi^{*}))(xy)-\mu(y)(f^{e}-\frac{c_{1}c_{2}}{4}(\chi+\chi^{*}))(\sigma(y)x)=0,\end{equation*}
for all $x,y\in S$. Hence, the function
$\theta:=f^{e}-\frac{c_{1}c_{2}}{4}(\chi+\chi^{*})$ belongs to
$\mathcal{N}_{\mu}(\sigma,S)$ and we obtain
\begin{equation}\label{eq7-23}f^{e}=\theta+\frac{c_{1}c_{2}}{4}(\chi+\chi^{*}).\end{equation}
On the other hand, by substituting (\ref{eq7-18}) and (\ref{eq7-21})
in (\ref{eq72-3}) we get that
\begin{equation*}\begin{split}&f^{o}(xy)-\mu(y)f^{o}(\sigma(y)x)=\frac{c_{1}c_{3}}{4}[\chi(x)+\mu(x)\chi(\sigma(x))][\chi(y)-\mu(y)\chi(\sigma(y))]\\
&=\frac{c_{1}c_{3}}{4}[\chi(xy)-\mu(y)\chi(\sigma(y)x)+\mu(x)\chi(\sigma(x)y)-\mu(xy)\chi(\sigma(xy))]\\
&=\frac{c_{1}c_{3}}{4}(\chi-\chi^{*})(xy)-\mu(y)\frac{c_{1}c_{3}}{4}(\chi-\chi^{*})(\sigma(y)x),\end{split}\end{equation*}
for all $x,y\in S$. So that
\begin{equation*}(f^{o}-\frac{c_{1}c_{3}}{4}(\chi-\chi^{*}))(xy)-\mu(y)(f^{o}-\frac{c_{1}c_{3}}{4}(\chi-\chi^{*}))(\sigma(y)x)=0,\end{equation*}
for all $x,y\in S$. Hence the function
$f^{o}-\frac{c_{1}c_{3}}{4}(\chi-\chi^{*})$ belongs to
$\mathcal{N}_{\mu}(\sigma,S)$. As
$f^{o}-\frac{c_{1}c_{3}}{4}(\chi-\chi^{*})$ is an odd function we
derive, according to Lemma \ref{lem621}(2), that
$f^{o}-\frac{c_{1}c_{3}}{4}(\chi-\chi^{*})=0$ and then
\begin{equation*}f^{o}=\frac{c_{1}c_{3}}{4}(\chi-\chi^{*}).\end{equation*}
Combining this with (\ref{eq7-22}) and using the fact that
$f=f^{e}+f^{o}$ we obtain
\begin{equation*}f=\theta+\frac{c_{1}}{2}[c_{2}\frac{\chi+\chi^{*}}{2}+c_{3}\frac{\chi-\chi^{*}}{2}],\end{equation*}
which gives with (\ref{eq7-18}) and (\ref{eq7-22}), by putting
$\alpha:=\frac{c_{1}c_{2}}{2}$ and $\beta:=\frac{c_{1}c_{3}}{2}$,
the result in
part (3) of Theorem \ref{thm63}.\\
(ii)
\begin{equation}\label{eq7-24}h=\chi\,A\quad\text{on}\quad S\setminus I_{\chi}
\quad\text{and}\quad h=0\quad\text{on}\quad I_{\chi}\end{equation}
and
\begin{equation}\label{eq7-25}m=\chi(1+cA)\quad\text{on}\quad S\setminus I_{\chi}
\quad\text{and}\quad m=0\quad\text{on}\quad I_{\chi},\end{equation}
where $c\in\mathbb{C}$ is a constant, $\chi:S\to\mathbb{C}$ is a
nonzero multiplicative function and $A:S\setminus
I_{\chi}\to\mathbb{C}$ is a nonzero additive function such that
$\chi^{*}=\chi$ and $A\circ\sigma=-A$. As in (ii) of the proof of
Proposition \ref{prop61} we have $\sigma(S\setminus
I_{\chi})=S\setminus I_{\chi}$ and $\sigma(I_{\chi})=I_{\chi}$.
Hence, by a small computation using that $\chi^{*}=\chi$ and
$A\circ\sigma=-A$, we get from (\ref{eq7-25}) that $m^{e}=\chi$. So,
in view of (\ref{eq7-13}) and (\ref{eq7-17}), we get that
\begin{equation}\label{eq7-26}g^{o}=c_{2}\chi\,A\quad\text{on}\quad S\setminus I_{\chi}
\quad\text{and}\quad g^{o}=0\quad\text{on}\quad
I_{\chi},\end{equation} and
\begin{equation}\label{eq7-27}g^{e}=c_{3}\chi,\end{equation} where $c_{2}:=b\in\mathbb{C}\setminus\{0\}$ is a
constant. Hence
\begin{equation}\label{eq7-28}g=c_{3}\chi+c_{2}\chi\,A\quad\text{on}\quad S\setminus I_{\chi}
\quad\text{and}\quad g=0\quad\text{on}\quad I_{\chi}.\end{equation}
Recall that $\sigma(S\setminus I_{\chi})=S\setminus I_{\chi}$ and
$\sigma(I_{\chi})=I_{\chi}$. Then
\par on the subsemigroup $S\setminus I_{\chi}$, we get, by using (\ref{eq7-26}),
(\ref{eq7-24}) and (\ref{eq72-2}), that
\begin{equation*}\begin{split}&f^{e}(xy)-\mu(y)f^{e}(\sigma(y)x)=c_{2}\chi(xy)A(x)A(y)\\
&=\frac{c_{2}}{4}[A^{2}(xy)-A^{2}(\sigma(y)x)]\chi(xy)\\
&=\frac{c_{2}}{4}\chi(xy)A^{2}(xy)-\frac{c_{2}}{4}\mu(y)\chi(\sigma(y)x)A^{2}(\sigma(y)x),\end{split}\end{equation*}
for all $x,y\in S\setminus I_{\chi}$, which implies that
$$(f^{e}-\frac{c_{2}}{4}\chi\,A^{2})(xy)-\mu(y)(f^{e}-\frac{c_{2}}{4}\chi\,A^{2})(\sigma(y)x)=0$$
for all $x,y\in S\setminus I_{\chi}$.\\
Defining $\theta_{1}:=f^{e}-\frac{c_{2}}{4}\chi\,A^{2}$ we see that
$\theta_{1}\in\mathcal{N}_{\mu}(\sigma,S\setminus I_{\chi})$ and
\begin{equation}\label{eq-29}f^{e}=\theta_{1}+\frac{c_{2}}{4}\chi\,A^{2}\quad\text{on}\quad S\setminus I_{\chi}.\end{equation}
Similarly, by using (\ref{eq7-27}), (\ref{eq7-24}) and
(\ref{eq72-3}), we get that
\begin{equation*}\begin{split}&f^{o}(xy)-\mu(y)f^{o}(\sigma(y)x)=c_{}\chi(xy)A(y)\\
&=\frac{c_{3}}{2}\chi(xy)[A(x)+A(y)-A(\sigma(y)x)]\\
&=\frac{c_{3}}{2}\chi(xy)A(xy)-\frac{c_{3}}{2}\mu(y)\chi(\sigma(y)x)A(\sigma(y)x),\end{split}\end{equation*}
for all $x,y\in S\setminus I_{\chi}$, which implies that
$$(f^{o}-\frac{c_{3}}{2}\chi\,A)(xy)-\mu(y)(f^{o}-\frac{c_{3}}{2}\chi\,A)(\sigma(y)x)=0$$
for all $x,y\in S\setminus I_{\chi}$. Hence
$f^{o}-\frac{c_{3}}{2}\chi\,A\in\mathcal{N}_{\mu}(\sigma,S\setminus
I_{\chi})$. As $\chi^{*}=\chi$ and $A\circ\sigma=-A$ we get that the
function $f^{o}-\frac{c_{3}}{2}\chi\,A$ is odd on the subsemigroup
$S\setminus I_{\chi}$, which is also generated by its squares. Thus,
according to Lemma \ref{lem621}(2), $f^{o}-\frac{c_{3}}{2}\chi\,A=0$
on $S\setminus I_{\chi}$ and then
\begin{equation}\label{eq-30}f^{o}=\frac{c_{3}}{2}\chi\,A\quad\text{on}\quad S\setminus I_{\chi}.\end{equation}
Hence, from (\ref{eq-29}) and (\ref{eq-30}) we deduce that
\begin{equation}\label{eq-31}f=\theta_{1}+\frac{c_{3}}{2}\chi\,A+\frac{c_{2}}{4}\chi\,A^{2}\quad\text{on}\quad S\setminus I_{\chi}.\end{equation}
\par On the subsemigroup $I_{\chi}$ we have $h=0$ by (\ref{eq7-24}).
Then we get for $x\in S$ and $y\in I_{\chi}$ that
$$f(xy)-\mu(y)f(\sigma(y)x)=g(x)h(y)=g(x)\cdot0=0,$$
which implies that
\begin{equation}\label{eq-311}\theta_{2}:=f|_{I_{\chi}}\in\mathcal{N}_{\mu}(\sigma,I_{\chi}).\end{equation} Defining
$\theta:=\theta_{1}\cup\theta_{2}$ we see that
$\theta\in\mathcal{N}_{\mu}(\sigma,S)$. Indeed, for $x,y\in S$,
\par If $xy\in S\setminus I_{\chi}$, then $x,y\in S\setminus
I_{\chi}$ and $\sigma(y)x\in S\setminus I_{\chi}$ because
$\chi^{*}=\chi$. So that
$$\theta(xy)-\mu(y)\theta(\sigma(y)x)=\theta_{1}(xy)-\mu(y)\theta_{1}(\sigma(y)x)=0.$$
\par If $xy\in I_{\chi}$, then $x\in I_{\chi}$ or $y\in I_{\chi}$,
 and $\sigma(y)x\in I_{\chi}$ because $\chi^{*}=\chi$. So, taking (\ref{eq7-24}), (\ref{eq7-28}) and (\ref{eq-311}), we get that
\begin{equation*}\begin{split}&\theta(xy)-\mu(y)\theta(\sigma(y)x)=\theta_{2}(xy)-\mu(y)\theta_{2}(\sigma(y)x)\\&=f(xy)-\mu(y)f(\sigma(y)x)=g(x)h(y)=0.\end{split}\end{equation*}
Now, by using (\ref{eq-31}) and (\ref{eq-311}), we obtain
\begin{equation}\label{eq7-35}f=\theta+\frac{c_{3}}{2}\chi\,A+\frac{c_{2}}{4}\chi\,A^{2}\quad\text{on}\quad S\setminus I_{\chi}
\quad\text{and}\quad f=\theta\quad\text{on}\quad
I_{\chi}.\end{equation} Combining (\ref{eq7-24}), (\ref{eq7-28}) and
(\ref{eq7-35}), and putting $\alpha:=c_{3}$ and $\beta:=c_{2}$, we obtain part (4) of Theorem \ref{thm63}.\\
\underline{Case B}: $b=0$. In this case we have, taking
(\ref{eq7-13}) into account, $g^{o}=0$. So, the functional equation
(\ref{eq72-2}) becomes
$$f^{e}(xy)-\mu(y)f^{e}(\sigma(y)x)=0,$$
for all $x,y\in S$. Hence, by defining $\theta:=f^{e}$ we have
\begin{equation}\label{eq7-36}f^{e}=\theta,\end{equation}
with $\theta\in\mathcal{N}_{\mu}(\sigma,S)$.\\
Moreover $g=g^{e}+g^{o}=g^{e}$. Then $g$ is an even function, i.e.,
$\mu g\circ\sigma=g$. Hence the identity (\ref{eq72-1}) reduces to
\begin{equation}\label{eq7-37}h(xy)=g(x)l(y)-g(y)l^{*}(x),\end{equation}
for all $x,y\in S$. According to Lemma \ref{lem62}(2) $h$ is
central, so (\ref{eq7-37}) implies that
$$g(x)l(y)-g(y)l^{*}(x)=g(y)l(x)-g(x)l^{*}(y),$$
for all $x,y\in S$. So,
$$g(x)[l(y)+l^{*}(y)]=g(y)[l(x)+l^{*}(x)],$$
for all $x,y\in S$, which implies that the functions $g$ and $l^{e}$
are linearly dependent. Since $g\neq0$ we deduce that there exists a
constant $c\in\mathbb{C}$ such that
\begin{equation}\label{eq7-38}l^{e}=cg.\end{equation}
Now, let $x,y\in S$ be arbitrary. By applying (\ref{eq7-37}) to the
pair $(\sigma(y),x)$, multiplying the identity obtained by $\mu(y)$
and taking into account that $\mu g\circ\sigma=g$ and
$\mu(y\sigma(y))=1$, we get that
\begin{equation}\label{eq7-39}\mu(y)h(\sigma(y)x)=g(y)l(x)-g(x)l(y).\end{equation}
By subtracting (\ref{eq7-39}) from (\ref{eq7-37}) we obtain
\begin{equation*}\begin{split}&h(xy)-\mu(y)h(\sigma(y)x)=g(x)l(y)-g(y)l^{*}(x)-g(y)l(x)+g(x)l(y)\\
&=2g(x)l(y)-g(y)[l(x)+l^{*}(x)]\\
&=2g(x)l(y)-2g(y)l^{e}(x),\end{split}\end{equation*} which implies,
using (\ref{eq7-38}), that
\begin{equation*}\begin{split}&h(xy)-\mu(y)h(\sigma(y)x)=2g(x)l(y)-2cg(x)g(y)\\
&=2g(x)[l(y)-cg(y)]=2g(x)[l(y)-l^{e}(y)].\end{split}\end{equation*}
So, $x$ and $y$ being arbitrary, the triple $(h,2g,l^{o})$ satisfies
the functional equation (\ref{EQ1}), i.e.,
\begin{equation}\label{eq7-40}h(xy)-\mu(y)h(\sigma(y)x)=2g(x)l^{o}(y),\end{equation}
for all $x,y\in S$.\\
By using the fact that $\mu(x\sigma(x))=1$ for all $x\in S$, $\mu
f^{o}\circ\sigma=-f^{o}$ and $g=g^{e}$ the functional equation
(\ref{eq72-3}) becomes
\begin{equation}\label{eq7-41}f^{o}(xy)+\mu(x)f^{o}(y\sigma(x))=g(x)h(y)\end{equation}
for all $x,y\in S$. Since $h\neq0$ there exists $y_{0}\in S$ such
that $h(y_{0})\neq0$. Let $x,y\in S$ be arbitrary. When we apply
(\ref{eq7-41}) to the pairs $(xy,y_{0})$ and $(x\sigma(y),y_{0})$ we
deduce that
$$g(xy)=\frac{1}{h(y_{0})}[f^{o}(xyy_{0})+\mu(xy)f^{o}(y_{0}\sigma(x)\sigma(y))]$$
and
$$g(x\sigma(y))=\frac{1}{h(y_{0})}[f^{o}(x\sigma(y)y_{0})+\mu(x\sigma(y))f^{o}(y_{0}\sigma(x)y)],$$
which implies, taking into account that $\mu$ is multiplicative and
$\mu(y\sigma(y))=1$, that
\begin{equation}\label{eq7-42}\begin{split}&g(xy)+\mu(y)g(x\sigma(y))=\frac{1}{h(y_{0})}[f^{o}(xyy_{0})+\mu(y)f^{o}(x\sigma(y)y_{0})\\&+
\mu(xy)f^{o}(y_{0}\sigma(x)\sigma(y))+\mu(x)f^{o}(y_{0}\sigma(x)y)].\end{split}\end{equation}
The identity (\ref{eq72-4}) shows that $f^{o}$ is central. Hence, by
applying (\ref{eq7-41}), we get that
\begin{equation}\label{eq7-43}f^{o}(xyy_{0})+\mu(y)f^{o}(x\sigma(y)y_{0})=f^{o}(yy_{0}x)+\mu(y)f^{o}(y_{0}x\sigma(y))=g(y)h(y_{0}x).\end{equation}
Similarly,
\begin{equation*}\begin{split}&\mu(xy)f^{o}(y_{0}\sigma(x)\sigma(y))+\mu(x)f^{o}(y_{0}\sigma(x)y)\\
&=\mu(x)[f^{o}(yy_{0}\sigma(x))+\mu(y)f^{o}(y_{0}\sigma(x)\sigma(y))].\end{split}\end{equation*}
So, by using (\ref{eq7-41}), we get that
\begin{equation}\label{eq7-44}\mu(xy)f^{o}(y_{0}\sigma(x)\sigma(y))+\mu(x)f^{o}(y_{0}\sigma(x)y)=\mu(x)g(y)h(y_{0}\sigma(x)).\end{equation}
From (\ref{eq7-42}), (\ref{eq7-43}) and (\ref{eq7-44}) we obtain
\begin{equation}\label{eq7-45}g(xy)+\mu(y)g(x\sigma(y))=\frac{1}{h(y_{0})}g(y)[h(y_{0}x)+\mu(x)h(y_{0}\sigma(x))].\end{equation}
Moreover, by applying (\ref{eq7-37}) to the pairs $(y_{0},x)$ and
$(y_{0},\sigma(x))$ we get that
\begin{equation*}h(y_{0}x)=g(y_{0})l(x)-g(x)l^{*}(y_{0})\end{equation*}
and
\begin{equation*}h(y_{0}\sigma(x))=g(y_{0})l(\sigma(x))-g(\sigma(x))l^{*}(\sigma(y_{0})).\end{equation*}
By adding the first identity above to the second one multiplied  by
$\mu(x)$, and seeing that $g^{*}(x)=g(x)$, we obtain
\begin{equation*}\begin{split}&h(y_{0}x)+\mu(x)h(y_{0}\sigma(x))=g(y_{0})[l(x)+l^{*}(x)]-2l^{*}(\sigma(y_{0}))g(x)\\
&=2g(y_{0})l^{e}(x)-2l^{*}(\sigma(y_{0}))g(x),\end{split}\end{equation*}
which implies, taking (\ref{eq7-38}) into account, that
\begin{equation*}\begin{split}&h(y_{0}x)+\mu(x)h(y_{0}\sigma(x))=2cg(y_{0})g(x)-2l^{*}(\sigma(y_{0}))g(x)\\
&=2g(x)[cg(y_{0})-l^{*}(y_{0})]\\
&=2g(x)[l^{e}(y_{0})-l^{*}(y_{0})].\end{split}\end{equation*}Thus
\begin{equation*}h(y_{0}x)+\mu(x)h(y_{0}\sigma(x))=2l^{o}(y_{0})g(x).\end{equation*}
Combining this with (\ref{eq7-45}) we get that
\begin{equation}\label{eq7-46}g(xy)+\mu(y)g(x\sigma(y))=\lambda g(x)g(y)\end{equation}
for all $x,y\in S$, where
$\lambda:=\frac{2l^{o}(y_{0})}{h(y_{0})}$.\\
On the other hand, since $g\neq0$ there exists $x_{0}\in S$ such
that $g(x_{0})\neq0$. Let $x,y\in S$ be arbitrary. By applying
(\ref{eq7-41}) to the pair $(x_{0},xy)$ and using that $f^{o}$ is
central we get that
$$2h(xy)=\frac{1}{g(x_{0})}[2f^{o}(xyx_{0})+2\mu(x_{0})f^{o}(xy\sigma(x_{0}))].$$
By applying the identity (\ref{eq72-4}) to the pairs $(x,yx_{0})$
and $(x,y\sigma(x_{0}))$, and using that $g^{e}=g$ and Lemma
\ref{lem62}(2), the identity above gives
\begin{equation*}\begin{split}&2h(xy)=\frac{1}{g(x_{0})}[g(x)h(yx_{0})+h(x)g(yx_{0})+\mu(x_{0})g(x)h(y\sigma(x_{0}))\\
&+\mu(x_{0})h(x)g(y\sigma(x_{0}))]\\
&=\frac{1}{g(x_{0})}\{g(x)[h(yx_{0})+\mu(x_{0})h(y\sigma(x_{0}))]+h(x)[g(yx_{0})+\mu(x_{0})g(y\sigma(x_{0}))]\}\\
&=\frac{1}{g(x_{0})}\{g(x)[h(yx_{0})+\mu(y)\mu(x_{0}\sigma(y))h\circ\sigma(x_{0}\sigma(y))]+h(x)[g(yx_{0})+\mu(x_{0})g(y\sigma(x_{0}))]\}\\
&=\frac{1}{g(x_{0})}\{g(x)[h(yx_{0})-\mu(y)h(x_{0}\sigma(y))]+h(x)[g(yx_{0})+\mu(x_{0})g(y\sigma(x_{0}))]\},\end{split}\end{equation*}
which implies, by using (\ref{eq7-40}) and (\ref{eq7-46}), that
\begin{equation*}\begin{split}&2h(xy)=\frac{1}{g(x_{0})}[2g(x_{0})g(x)l^{o}(y)+\lambda g(x_{0})h(x)g(y)].\end{split}\end{equation*}
So, $x$ and $y$ being arbitrary, the pair $(h,\frac{\lambda}{2}g)$
satisfies the sine addition law
\begin{equation}\label{eq7-47}h(xy)=h(x)(\frac{\lambda}{2}g(y))+(\frac{\lambda}{2}g(x))h(y)\end{equation}
for all $x,y\in S$. Notice that $\lambda\neq0$ because $h\neq0$ and
$S$ is generated by its squares. According to \cite[Proposition
1.1]{Bouikhalene and Elqorachi}, and proceeding as (i) and (ii) of
the proof of Proposition \ref{prop61}, we get that the pair $(h,g)$
falls into two categories:\\
(i)
\begin{equation}\label{eq7-48}h=c_{1}\frac{\chi-\chi^{*}}{2}\end{equation}
and
\begin{equation}\label{eq7-49}g=c_{2}\frac{\chi+\chi^{*}}{2},\end{equation}
where $c_{1},c_{2}\in\mathbb{C}\setminus\{0\}$ are constants and
$\chi:S\to\mathbb{C}$ is a multiplicative function such that
$\chi^{*}\neq\chi$.\\
By substituting (\ref{eq7-48}) and (\ref{eq7-49}) in the functional
equation (\ref{eq72-3}), and taking into account that $g^{e}=g$, and
proceeding exactly as in part (i) of Case A to compute $f^{o}$, we
obtain
\begin{equation*}f^{o}=\frac{c_{1}c_{2}}{2}\frac{\chi-\chi^{*}}{2},\end{equation*}
which gives with (\ref{eq7-36})
\begin{equation}\label{eq7-50}f=\theta+\frac{c_{1}c_{2}}{2}\frac{\chi-\chi^{*}}{2}.\end{equation}
By putting $\beta:=\frac{c_{1}c_{2}}{2}\in\mathbb{C}\setminus\{0\}$, we get from (\ref{eq7-48}), (\ref{eq7-49}) and (\ref{eq7-50})
the part (3) of Theorem \ref{thm63} corresponding to $\alpha=0$.\\
(ii) \begin{equation}\label{eq7-51}h=\chi\,A\quad\text{on}\quad
S\setminus I_{\chi} \quad\text{and}\quad h=0\quad\text{on}\quad
I_{\chi}\end{equation} and
\begin{equation}\label{eq7-52}g=c_{1}\chi\,A\quad\text{on}\quad S\setminus I_{\chi}
\quad\text{and}\quad g=0\quad\text{on}\quad I_{\chi},\end{equation}
where $c_{1}\in\mathbb{C}\setminus\{0\}$ is a constant,
$\chi:S\to\mathbb{C}$ is a nonzero multiplicative function and
$A:S\setminus I_{\chi}\to\mathbb{C}$ is a nonzero additive function
such that $\chi^{*}=\chi$ and
$A\circ\sigma=-A$.\\
By using (\ref{eq7-51}), (\ref{eq7-52}) and similar computations to
the ones made to compute $f^{o}$ on $S\setminus I_{\chi}$ and
$I_{\chi}$ in part (ii) of Case A, we obtain
\begin{equation*}f^{o}=\frac{c_{1}}{2}\chi\,A\quad\text{on}\quad S\setminus I_{\chi}
\quad\text{and}\quad f^{o}=0\quad\text{on}\quad
I_{\chi},\end{equation*} which gives with (\ref{eq7-36})
\begin{equation}\label{eq7-53}f=\theta+\frac{c_{1}}{2}\chi\,A\quad\text{on}\quad S\setminus I_{\chi}
\quad\text{and}\quad f=\theta\quad\text{on}\quad
I_{\chi}.\end{equation} By putting
$\alpha:=\frac{c_{1}}{2}\in\mathbb{C}\setminus\{0\}$, we get, from
(\ref{eq7-51}), (\ref{eq7-52}) and (\ref{eq7-53}), part (4) of
Theorem \ref{thm63} corresponding to $\beta=0$. This completes the
proof of Theorem \ref{thm63}.
\end{proof}
\begin{rem} Let $\chi:S\to\mathbb{C}$ be a nonzero multiplicative function. If $S=\{xy\,|\,x,y\in S\}$
then $S\setminus I_{\chi}=\{xy\,|\,x,y\in S\setminus I_{\chi}\}$. If
$S$ is a regular semigroup so is the subsemigroup $S\setminus
I_{\chi}$. Then, in view of Remark \ref{rem1} and Remark \ref{rem2},
the proof of Theorem \ref{thm63} works for each type of semigroups.
It follows that our results are valid if $S=\{xy\,|\,x,y\in S\}$,
and if $S$ is a regular semigroup.
\end{rem}

\end{document}